\documentclass[12pt]{article}
\usepackage{mathrsfs}
\usepackage{amsmath}
\usepackage{amsthm}
\usepackage{bbm}
\usepackage{subfloat}
\usepackage{subfigure}
\usepackage{fancyhdr,graphicx}
\usepackage{amsfonts}
\usepackage{amssymb}
\usepackage{latexsym,bm}
\usepackage{newlfont}

\newtheorem{thm}{Theorem}[section]
\newtheorem{Lemma}[thm]{Lemma}
\newtheorem{cor}[thm]{Corollary}

\newtheorem{observ}[thm]{Observation}
\newtheorem{Alg}[thm]{Algorithm}

\usepackage{latexsym, bm}
\title{On the anti-forcing number of  fullerene graphs\footnote{This work was supported by NSFC (Grant No. 11371180).
       }}

\author{Qin Yang, Heping Zhang\footnote{The corresponding author.}
\\\small{School of Mathematics and Statistics, Lanzhou
University, Lanzhou, Gansu 730000, P. R. China.}
\\\small{E-mail addresses: yangqin4587@126.com, zhanghp@lzu.edu.cn
}
\\Yuqing Lin
\\\small{School of Electrical Engineering and Computer Science, The University of Newcastle, Australia}
\\\small{E-mail addresses: yuqing.lin@newcastle.edu.au}
}

\date{}

\begin{document}

\maketitle \noindent \textbf{Abstract:} The anti-forcing number of a
connected graph $G$ is the smallest number of edges such that the
remaining graph obtained by deleting these edges has a unique
perfect matching. In this paper, we show that the
anti-forcing number of every fullerene has at least four. We give a procedure to construct all fullerenes whose  anti-forcing numbers achieve the lower bound four. Furthermore, we show that, for every
even $n\geq20$ ($n\neq22,26$), there exists a
fullerene with  $n$ vertices that has the anti-forcing number four, and the fullerene with 26 vertices has the anti-forcing number five.

\noindent \textbf{Keywords:} Fullerene graph; Perfect matching; Anti-forcing number; Forcing number

\section{Introduction}

\setlength{\unitlength}{1cm}

A {\em fullerene  graph } (simply fullerene) is a cubic 3-connected plane graph
with only pentagonal faces (exact 12 of them by Euler's polyhedron formula) and hexagonal faces. It is a molecular graph of novel spherical carbon clusters called fullerenes \cite{fowler}.    The first fullerene molecule  C$_{60}$ was  discovered in 1985 by Kroto et al. \cite{c60}. It is well known that a fullerene graph on $n$
vertices exists for every even $n\geq20$ except $n=22$ \cite{GM}.

A set of independent edges of a  graph $G$ is called a
\textit{matching} of $G$. A matching $M$ of $G$ is called \textit{perfect matching} (or Kekul\'{e}
structure in chemical literature) if every vertex of $G$ is incident with exactly one edge in $M$.
 Kekul\'{e} structure plays a very important role in analysis of the property of benzenoid hydrocarbons, fullerenes
and other carbon cages. 

Let $G$ be a graph with a perfect matching  $M$.  A set $S\subseteq M$ is called a {\em forcing set} of
$M$ if $S$ cannot be contained in another perfect matching of $G$ other than $M$. The {\em forcing number} (or
innate degree of freedom) of $M$ is defined as the minimum size of
all forcing sets of $M$, denoted by $f(G,M)$ \cite{F. Harary,D.J. Klein}.
The minimum forcing number of $G$ is the minimal value of the forcing numbers
of all perfect matchings of $G$, denoted by $f(G)$.  Zhang,  Ye
and Shiu \cite{fcf} proved that the minimum forcing number
of fullerenes has a lower bound 3 and there are infinitely many fullerenes achieving this bound.

Let $G$  be a  graph with vertex set $V(G)$ and edge set $E(G)$.  For $S \subseteq E(G)$, let $G-S$ denote
the graph obtained by removing $S$ from $G$. Then $S$ is called an
{\em anti-forcing set} if $G-S$ has a unique perfect matching. The cardinality of a smallest
anti-forcing set is called the {\em anti-forcing number} of $G$, denoted by $af(G)$.

D. Vuki\v{c}evi\'{c} and N.
Trinajsti\'{c} \cite{afcben,afcata} recently introduced  the anti-forcing number of graphs and determined the anti-forcing numbers of parallelogram benzenoid
and cata-condensed benzenoids. In fact,    X. Li \cite{xli} had showed that a benzenoid with a forcing single edge (equivalently, it has the anti-forcing number one) if and only if it is a truncated parallelogram before it. In this paper, we prove that the  anti-forcing
number of fullerenes has a lower bound  4. Then we present an approach to generate  all fullerenes which
achieve the lower bound 4 for the anti-forcing number. Furthermore, we demonstrate how to construct at least one fullerene $F_{n}$ with $n$ vertices such that
$af(F_{n})=4$ for every even $n\geq20$ except $n=22$ and $26$.

\section{Basic definitions and preliminaries}

\setlength{\unitlength}{1cm}

Firstly, we summarize some known results on the extendability and connectivity of fullerene graphs, which will be used in proving our main results.

A connected graph $G$ with at least $2(k+1)$ vertices is said to be
\textbf{$k$-$extendable$} if $G$ has a perfect matching and any $k$ disjoint edges of $G$ belong to
a perfect matching of $G$.

\begin{thm}
(\cite{2extendable}) \label{2-entend} Every fullerene graph is 2-extendable.
\end{thm}

Let $G$ be a connected graph.  For nonempty subsets $X, Y$ of $V(G)$, let $[X,Y]$ denote the
set of edges of $G$ that each has one end-vertex  in $X$ and the other in $Y$.
If $\overline{X}=V(G)\setminus X\not=\emptyset$, then  $[X, \overline{X}]$ is called an {\em edge-cut} of $G$, and {\em $k$-edge-cut} whenever $|[X, \overline{X}]|=k$.  An edge-cut $S$ of $G$ is {\em cyclic} if at least two components of $G- S$ contain a cycle.
%

The cardinality of the smallest  cyclic edge-cut of $G$ is called the {\em cyclic edge connectivity} of $G$, denoted by $c\lambda(G)$.
We call an edge-cut {\em trivial} if its edges are incident with the
same vertex. We call a  cyclic $k$-edge-cut {\em trivial} if one of the resulting components
induces a single $k$-cycle.

\begin{thm}\label{clic is 5}
\cite{kcyclic conn1,kcyclic conn2} Every fullerene graph has the
cyclic edge connectivity 5.
\end{thm}
The theorem together with 3-connectivity imply that every fullerene graph has the girth 5 ( the length of  a shortest cycle), and each of all the 5-cycles and 6-cycles of a fullerene graph bound a face. So pentagonal face and hexagonal face in a fullerene coincide with pentagon and hexagon respectively. A pentacap is a graph consisting of 6 pentagons, as shown in
Fig. 1 (left). F. Kardo\v{s} and R. \v{S}krekovski \cite{kardos}, and K. Kutnar and D. Maru\v{s}i\v{c} \cite{cbgkehua} independently proved that there is
only one class of fullerenes which admit nontrivial cyclic 5-edge-cut,
as shown in Fig. 1 (right).

\begin{figure}[htbp]
\begin{center}
\includegraphics[totalheight=4.0 cm]{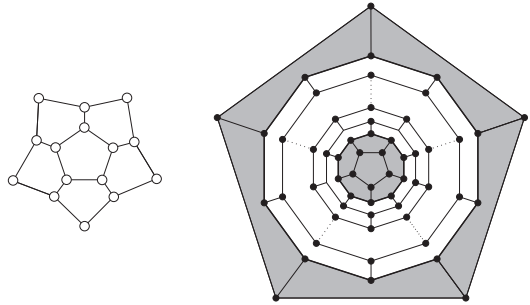}
\caption{\label{1}\small{A pentacap (left) and a class of fullerenes
which admit nontrivial cyclic 5-edge-cut (right).}}
\end{center}
\end{figure}

\begin{thm}
\cite{kardos, cbgkehua} Let F be a fullerene admitting a nontrivial cyclic
5-edge-cut. Then F contains a pentacap, and more precisely, F
contains two disjoint antipodal pentacaps.
\end{thm}

From the structure of fullerenes admitting a nontrivial cyclic
5-edge-cut, we have the following observations.

\begin{observ}\label{nontrivial cyclic 5-edge-cut cor}
Let F be a fullerene with a nontrivial cyclic 5-edge-cut S. Then each
component of  $F-S$ is 2-connected. Furthermore, each component
deleting one 2-degree vertex is still 2-connected.
\end{observ}

\begin{observ}\label{the conment edges in cyclic 5-edge-cuts 1}
Let F be a fullerene with a nontrivial cyclic 5-edge-cut S. Then there is no common edge
for any two nontrivial cyclic 5-edge-cuts. And there exists at most one common edge for any
two cyclic 5-edge-cuts, one of them is nontrivial cyclic 5-edge-cut.
\end{observ}

We can  see the following basic fact:

\begin{Lemma}\label{3-edge-cut}
 Every 3-edge-cut of a fullerene graph $G$ is trivial.
\end{Lemma}
\begin{proof}
Let $E=[X,\overline{X}]$ be a 3-edge-cut of $G$, where
$\emptyset\neq X\subset V(G)$, $\overline{X}=V(G)\setminus
X$. Denote by $G[X]$ (resp. $G[\overline{X}]$) the graph induced by
$X$ (resp. $\overline{X}$) in $G$. Since $G$ is 3-connected, $G[X]$ and $G[\overline{X}]$ are components of $G-E$.

Suppose that $X$ and $\overline{X}$ both have at least two vertices. Then we have that $\frac{3|X|-3}{2}>\frac{2(|X|-1)}{2}=|X|-1$
and $\frac{3|\overline{X}|-3}{2}>\frac{2(|\overline{X}|-1)}{2}=|\overline{X}|-1$,
which implies that $G[X]$ and $G[\overline{X}]$ both contain a cycle, contradicting that the
cyclic edge connectivity of $G$ is 5. So  $X$ or
$\overline{X}$ is a singleton.
\end{proof}

Using this lemma, we could show the following result.

\begin{Lemma}\label{the conment edges in cyclic 5-edge-cuts 2}
 Let $F$ be a fullerene. Then there is  at most one common edge for any
two cyclic 5-edge-cuts.
\end{Lemma}

\begin{proof}
By Observation \ref{the conment edges in cyclic 5-edge-cuts 1}, we only need to show that
there is at most one common edge for any two trivial cyclic 5-edge-cuts.

To the contrary, suppose that there are two trivial cyclic 5-edge-cuts $S_{1}$ and $S_{2}$ which share at least two common edges which connect two pentagons  $P_{1}$ and $P_{2}$. Denote such common edges by $e_1$, $e_{2}$,\ldots,$e_k$, $2\le k\le 5$, in a consecutive order along the boundary of $P_1$ or $P_2$.  Then the plane subgraph $P_1\cup P_2+\{e_1,e_2,\ldots,e_k\}$ of $F$ has exactly $k$ faces except $P_1$ and $P_2$. Let $R_i$ denote  such a face such that two  consecutive edges $e_i$ and $e_{i+1}$ lie on its boundary $C_i$, for $1\leq i\le k$, where the subscript always modulo $k$.   Then $C_i$ has no chords (a chord of a cycle $C$ means an edge not in $C$ but both end-vertices in $C$). Since $R_i$ is not a face of $F$, $F$ has at least three edges issuing  from distinct vertices on $C_i$ into $R_i$.  This implies that in addition to the four endvertices of $e_i$ and $e_{i+1}$, $C_i$ has at least three additional vertices. Note $P_1$ and $P_2$ have totally ten vertices. We must have that $k=2$ and there are exactly three edges issuing  from  $C_i$ to the same vertex in $R_i$ for $i=1$ and $2$ by Lemma \ref{3-edge-cut} (see Fig. \ref{1}). Thus  $F$ has  12 vertices, contradicting  that  any fullerene has at least 20 vertices.
\end{proof}
\begin{figure}[htbp]
\begin{center}
\includegraphics[totalheight=3cm]{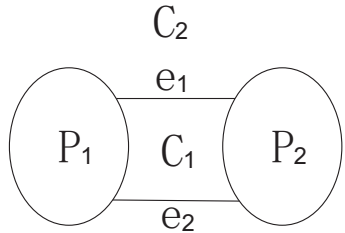}
\caption{\label{1}\small{The illustration for the proof of Lemma
\ref{the conment edges in cyclic 5-edge-cuts 2}.}}
\end{center}
\end{figure}

Similar to Lemma \ref{3-edge-cut} we have the following result.
\begin{Lemma}\label{4-edge-cut}
 Every 4-edge-cut of a fullerene graph $G$ isolates an edge.
\end{Lemma}
\begin{proof}Let $E=[X,\overline{X}]$ be a 4-edge-cut of $G$, where
$\emptyset\neq X\subset V(G)$, $\overline{X}=V(G)\setminus
X$. Denote by $G[X]$ (resp. $G[\overline{X}]$) the graph induced by
$X$ (resp. $\overline{X}$) in $G$. Since $G$ is 3-connected, $G[X]$ and $G[\overline{X}]$ are components of $G-E$.

If both $X$ and $\overline{X}$ have at least three vertices, then we have that $\frac{3|X|-4}{2}>|X|-1$
(similarly for  $\overline{X}$),
which implies that $G[X]$ and $G[\overline{X}]$ both contain a cycle, contradicting that $c\lambda(G)=5$. So  $X$ or
$\overline{X}$ has two vertices.
\end{proof}

An edge of a graph $G$ is said to be a {\em pendent edge} if it has an endvertex of degree one in $G$. Such an  endvertex is called a pendent vertex. An edge $e$ of $G$ is called a {\em bridge} if deleting $e$ from $G$
increases the number of components.

\begin{thm}\label{unique matching}
\cite{A.Kotzig} Let G be a connected graph with a unique perfect
matching. Then G has a bridge belonging to the perfect matching.
\end{thm}

From the theorem, it is clear that if a 2-edge connected graph contains a perfect matching,
then it has at least two perfect matchings.

\section{The lower bound of anti-forcing number of fullerenes}

\setlength{\unitlength}{1cm}

A lower bound of
the minimum forcing number of fullerenes is stated in the following.

\begin{thm}\label{focing matching}
\cite{fcf} Let F be a fullerene graph. Then $f(F)\geq 3$.
\end{thm}

Now we state a lower bound of anti-forcing number of
fullerenes as follows.

\begin{thm}\label{lower bond }
For a fullerene graph $F$,  $af(F)\geq 4$.
\end{thm}

\begin{proof}We just show  that $F$ does not contain an anti-forcing set of three distinct edges.
Suppose to the contrary that $F$  contains an anti-forcing set
$S=\{e_{1},e_{2},e_{3}\}$. Then
$F-S$ has a unique perfect matching $M$. There are two cases to be considered.

{\bf Case 1.} There exist two adjacent edges in $S$. Suppose that $e_{1}$
and $e_{2}$ are incident to the same end-vertex of an edge $e$ in $M$. Then $F$
has a cycle $C$ containing both $e$ and $e_{3}$ since $F$ is
3-connected. Since $F$ is cyclically 5-edge connected, we have that
$C$ has length at least 5. So there exists an edge $e'$ of $C$
such that $e$ and $e'$ are disjoint and $e'$ has a common
end-vertex with $e_{3}$. By the 2-extendability of $F$ (Theorem \ref{2-entend}) and $f(F)\geq
3$ (Theorem \ref{focing matching}), $F$ has at least 2 perfect matchings containing both $e$ and $e'$.
So $F-S$ has at least 2
perfect matchings,  a contradiction.

{\bf Case 2.} Any two edges in $S$ are not adjacent.  By
Theorem \ref{unique matching}, $F-S$ has a bridge $e$ in $M$. Let
$S':=S\cup \{e\}$. Then $F-S'$ is not connected. We claim that $S'$ is a minimal edge-cut of $F$. Otherwise,
some three edges in $S'$ form a trivial edge-cut by Lemma  \ref{3-edge-cut}, a contradiction. Let
$S'=[X,\overline{X}]$, where $\emptyset\neq X\subset V(F)$, and
$\overline{X}=V(F)\setminus X$. Since any two edges in $S$ are not
adjacent,  $|X|\geq 3$ and $|\overline X|\geq 3$. This contradicts Lemma \ref{4-edge-cut}.
\end{proof}

%
%
%
%

Now we show that this bound is sharp. In Fig. \ref{Example}, the remaining graphs
 by deleting four edges $\{e_{1},e_{2},e_{3},e_{4}\}$ from
$F_{20}$ (left) and $F_{24}$ (right) have only one perfect
matching. So both $af(F_{20})$ and $af(F_{24})$ equal 4.
\begin{figure}[htbp]
\begin{center}
\includegraphics[totalheight=5 cm]{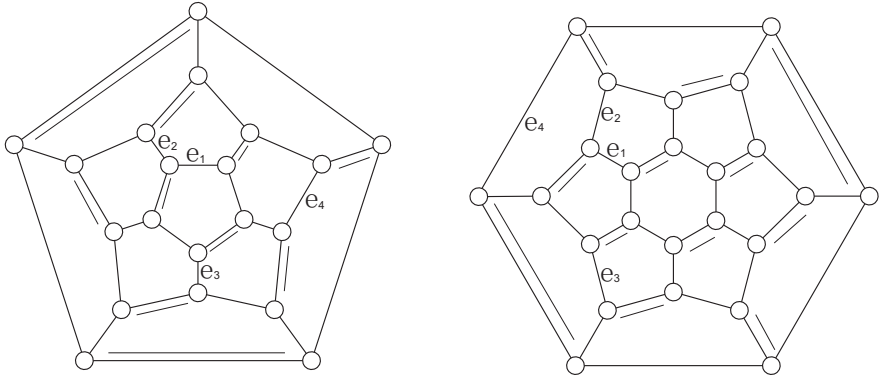}
\caption{\label{Example}\small{$F_{20}$ (left) and $F_{24}$ (right) have
an anti-forcing set of size  4.}}
\end{center}
\end{figure}

 Next, we study the structure of  anti-forcing sets of size  4.

\begin{thm}\label{anti set}
 Let $S$ be an anti-forcing set of a fullerene graph $F$ and $|S|=4$.  Then $S$ is not a matching, and
 $G[S]$ consists of  either
  a path of length 2 and two independent edges, or
  a path of length 3 and one independent edge.
 \end{thm}

\begin{proof}

We first claim that $S$ possesses two adjacent edges in $F$. Suppose to the contrary that $S$ is a matching. Since
$S$ is an anti-forcing set, we have that  $F-S$ has a unique perfect matching
$M$. By  Theorem \ref{unique matching}, $F-S$ has a bridge
$e=v_{1}v_{2}$ in $M$. Let $S':=S\cup \{e\}$. So $F-S'$ is disconnected.  Similar to the proof (Case 2) of Theorem \ref{lower bond } we can obtain that  $S'$ is a minimal 5-edge-cut. So there exists a proper nonempty subset $X$ of $V(F)$ such that $S'=[X,\overline{X}]$, where
$\overline{X}=V(F)\setminus X$. Since $S$ is a matching,  $|X|\geq 4$ and $|\overline X|\geq 4$. So
$\frac{3|X|-5}{2}-(|X|-1)=\frac{|X|-3}{2}>0$ (similarly for $\overline X$). Thus $S'$
is a cyclic 5-edge cut. Let $F_{1}=F[X]$ and $F_{2}=F[\overline{X}]$ denote the two components of $F-S'$ such that $v_{1}\in F_{1}$ and  $v_{2}\in F_{2}$.

If $S'$ is a nontrivial cyclic 5-edge-cut, then by Observation \ref{nontrivial cyclic 5-edge-cut cor},
$F_{1}$ is 2-connected, $v_1$ has degree 2 in $F_1$, and
$F_{1}-v_{1}$ is 2-connected. By Theorem \ref{unique matching} $F_{1}-v_{1}$ has more than one perfect matching, contradicting that $F-S$ has a unique perfect matching.  If $S'$ is a trivial cyclic 5-edge-cut, then one of  $F_1$ and $F_2$, say $F_2$,
must be a pentagon. 
Let $e_a$ and $e_b$ be two edges of $F_2$ belonging to $M$. Sine $F$ is 2-extendable and $f(F)\geq 3$, then $\{e_a,e_b\}$ is not a forcing set of $M$. So $\{e_a,e_b\}$ is a subset of at least two perfect matchings of $F$. Thus  $F-S$
has at least 2 perfect matchings, a contradiction. The claim is verified.

Now suppose $F[S]$ (the subgraph induced by  $S$) contains two edge-disjoint paths of length 2. It is clear that $F-S$ contains two pendent edges that  together are adjacent all edges in $S$.
We can see  that these two pendent edges belong to $M$ and form a forcing set of  $M$ in $F$, which contradicts that $f(F)\geq 3$. Noting that each vertex of $F[S]$ is of degree 1 or 2, $F[S]$ is either union of a path with length 2 and two paths with length 1 or the union of a path with length 3 and a path with length one.
\end{proof}

By Theorem \ref{anti set},  $F-S$ must contain pendent edges.
Since $F-S$ has a unique perfect matching $M$, the pendent edges must belong to the perfect matching.


\begin{thm}\label{F' structure}
Let F be a fullerene with $af(F)=4$ and $E_0=\{e_{1},e_{2},e_{3},e_{4}\}$
 an anti-forcing set of $F$. Let $F'$ be the remaining graph after deleting both end vertices of all the pendent edges from $F-E_0$ recursively, and $F''$  the subgraph
induced by all deleted vertices. If $F'=\emptyset$, then $F''=F$; If $F'\neq \emptyset$,
then $F'$ consists of two disjoint pentagons and  one edge between them (see Fig. \ref{fzuizhong}) and there are exactly 8 edges from $F'$ to $F''$.
\end{thm}

\begin{figure}[htbp]
\begin{center}
\includegraphics[totalheight=2 cm]{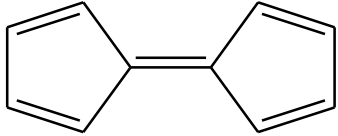}
\caption{\label{fzuizhong}\small{The structure of subgraph $F'$ with a unique Kekul\'e structure.}}
\end{center}
\end{figure}

\begin{proof}
Let $M$ be the unique perfect matching of $F-E_0$. By Theorem \ref{anti set}, there exist two adjacent edges in $E_0$, say $e_{1}$
and $e_{2}$. So there is one pendent edge $f_1=w_1z_1$ of  $F-E_0$  that is adjacent to  edges $e_{1}$ and $e_{2}$ in $E_0$.
Then $f_1$ is viewed as the first pendent edge, which must belong to the unique perfect matching $M$ of $F-E_0$. If $F-E_0-w_1-z_1$ has a pendent edge, then it is viewed as second pendent edge,  belonging to $M$. In general, let us define the following notations:

$X$: the set of edges of $F$ from $F'$ to $F''$,

$F'_{i}$: the remaining graph by deleting both end vertices of the former  $i$ pendent edges from $F-E_0$ (i.e. from the first to $i$-th pendent edges),

$F''_{i}$: the subgraph induced by all deleted vertices after deleting
both end vertices of the $i$-th pendent edge from $F-E_0$,

$X_{i}$: the set of edges in $F$ from $F'_{i}$ to $F''_{i}$.

If $F'\not=\emptyset$, then   from the above definition we have that $F'$ also has a unique prefect matching $M'=M|_{F'}$. By
Theorem \ref{unique matching}, there is a bridge $e'=v_{1}v_{2}$ of  $F'$ in $M'$. Then $F'$ is formed by  two disjoint subgraphs  $G_{1}$ and $G_{2}$ of $F'-e'$ connected by the edge $e'$
such that $v_{1}\in V(G_{1})$ and $v_{2}\in V(G_{2})$.
Since $F'$ has no pendent edges,  $G_{1}$ and $G_{2}$ both
contain cycles. The edges of $F$ from $F''$ to $G_{1}$ and $G_{2}$ are
denoted by $X_{G_1}$ and $X_{G_2}$, respectively (see
Fig. \ref{FFshiyitu}).
\begin{figure}[htbp]
\begin{center}
\includegraphics[totalheight=4 cm]{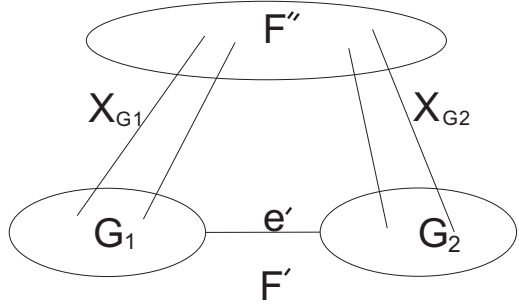}
\caption{\label{FFshiyitu}\small{The illustration for the proof of
Theorem \ref{F' structure}.}}
\end{center}
\end{figure}

{\em Claim 1}. $|X_i|\leq 8$ for each $i$.

Since  $F''_{1}$ is a complete graph with two vertices, we have that $|X_{1}|=4$.

Second pendent edge must exist, i.e. $F_1'$ has a pendent edge. If not, $F'=F_1'\not=\emptyset$. Since $\{e_1,e_2\}\subset X_1$, $[V(G_1),V(G_2)]$ has at most three edges in $F$, at most two of which belong to $E_0$. It follows that both  $X_{G_1}\cup [V(G_1),V(G_2)]$ and $X_{G_2}\cup [V(G_1),V(G_2)]$ would be  cyclic 5-edge-cuts of $F$, which have three edges in common,  contradicting Lemma \ref{the conment edges in cyclic 5-edge-cuts 2}.

Since there is no cycles of length 3 in fullerene,  the second
pendent edge must be adjacent to an edge in $E_0-\{e_{1},
e_{2}\}$, say $e_{3}$. By Theorem \ref{anti set}, the second pendent vertex is incident with $e_3$ and an edge in $X_1$ as well.
Then the $F''_{2}$ is a path of length 3, i.e. it consists of the first pendent edge and the second pendent edge and an edge connecting them.
So it is clear that $|X_{2}|=6$.

Similarly, $F'_2$ also has a pendent edge. If the third pendent edge is adjacent to $e_{4}$, then it is also adjacent to an edge in $X_2$, so $F''_3$ is connected. Hence $|X_3|\leq 3\times 6-2\times (6-1)=8$.  Further, if $F_i'$, $i\geq 3$, has a pendent edge, then the $(i+1)$-pendent edge  must be adjacent to two edges in $X_{i}$. So $|X_{i+1}|\leq |X_i|$. The induction procedure implies that we have $|X_i|\leq 8$ for all $i$.

\begin{figure}[htbp]
\begin{center}
\includegraphics[totalheight=3 cm]{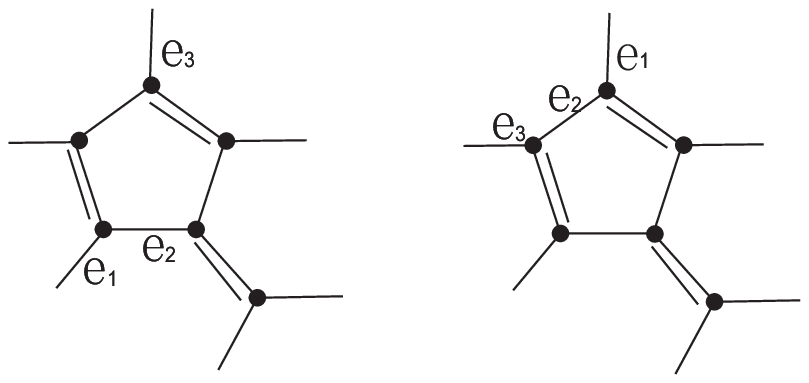}
\caption{\label{f3notcontaine4}\small{Two cases for the third
pendent edge not adjacent to $e_{4}$. }}
\end{center}
\end{figure}

If the third pendent edge is not adjacent with $e_{4}$, then it must be adjacent to two edges in $X_2$; both possible
cases of $F_3''$ are shown in Fig. \ref{f3notcontaine4}.
It is clear that  $|X_{3}|=6$. Similarly there exists a pendent edge of $F_3'$ by Lemma \ref{the conment edges in cyclic 5-edge-cuts 2}. From the structure of $F''_{3}$, we know
that the corresponding forth 1-degree vertex must be incident with $e_{4}$ and one edge of $X_3$. So $F_4''$ is connected and contains a cycle. Hence $|X_4|\leq 3\times 8-2\times 8=8$.    Similar to the above situation we can obtain  $|X_{i}|\leq 8$ for any $i$.
 Thus we have $|X|\leq 8$.

 By the above proof, we also obtain the following claim.


{\em Claim 2}.  $E_0\subseteq E(F'')\cup X$.

%
%


Now suppose that $F'\not=\emptyset$. We  will prove that $|X|=8$ and  $F'$ is formed by two
disjoint pentagons connected by exactly one edge.

By  Claims 1 and  2, we have that $|X|\leq 8$, $[V(G_1),V(G_2)]=\{e'\}$ and $\{e'\}\cup X_{G_1}$ and $\{e'\}\cup X_{G_2}$ are two cyclic edge-cuts of $F$.
Since $c\lambda(F)=5$,  $|X_{G_1}|+|X_{G_2}|=|X|\leq 8$. So $|X|=8$ and $\{e'\}\cup X_{G_1}$ and $\{e'\}\cup X_{G_2}$ are two cyclic 5-edge-cuts.

If $\{e'\}\cup X_{G_i}$ is nontrivial, $i=1$ or 2,  then by
Observation \ref{nontrivial cyclic 5-edge-cut cor}, $G_{i}-v_{i}$ is 2-connected and has
at least two perfect matchings, which  contradicts that $F'$ has a unique perfect matching. So
$\{e'\}\cup X_{G_1}$ and $\{e'\}\cup X_{G_2}$ are both trivial cyclic
5-edge-cuts. Hence both $G_1$ and $G_2$ are pentagons, and the required results follow.
\end{proof}

\section{Construction for all fullerenes  with   anti-forcing number 4}

In this section we  present a construction for generating all fullerene graphs  with anti-forcing number four.

Suppose that $F$ is a fullerene graph with an antiforcing set  $E_0=\{e_{1},e_{2},e_{3},e_{4}\}$.  By Theorem \ref{F' structure} with  notations in its proof,  we have that each $F_i'$, $1\leq i \leq 6$, has a pendent edge since $F$ has at least 20 vertices.  If the former three consecutive pendent edges  are all adjacent to edges of $E_0$, we must describe all possible structures of $F_3''$ induced by all endvertices of these edges. It is known that $F_2''$ is a path of length 3, and the third pendent vertex is incident with edge $e_4$ and has a neighbor in $F_2''$. So we can check that all cases ({\em configurations}) of $F_3''$ attached possible additional edges are showed in Fig. \ref{fdaofk} in a sense that in the same configuration edge set $\{e_{1},e_{2},e_{3},e_{4}\}$ can be chosen in a  different way.

\begin{figure}[htbp]
\begin{center}
\includegraphics[totalheight=6.6cm]{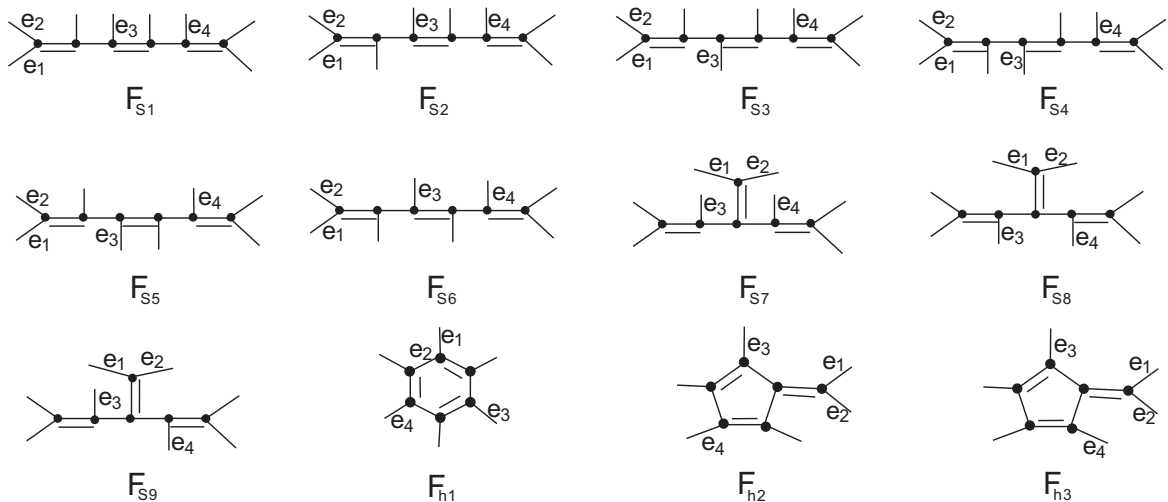}
\caption{\label{fdaofk}\small{All configurations of $F_3''$ that the  third pendent edge
is adjacent to  an edge in $E_0$.}}
\end{center}
\end{figure}

If the third pendent edge is  adjacent  to no edges of $E_0$, then it has two neighbors in $F''_2$ and the  fourth pendent edge must be adjacent to $e_4$. From two possible structures of $F_3''$ in Fig. \ref{f3notcontaine4} we can obtain  all cases (configurations) of $F_4''$ attached possible additional edges as shown in Fig. \ref{fdaofk2}.

\begin{figure}[htbp]
\begin{center}
\includegraphics[totalheight=6.6cm]{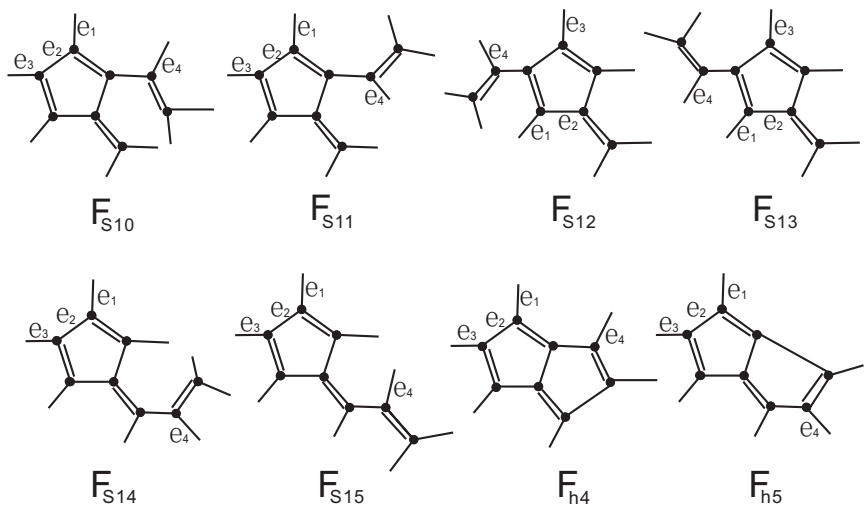}
\caption{\label{fdaofk2}\small{All configurations of $F_4''$ that the  third pendent edge
is adjacent to  no edge in $E_0$}}
\end{center}
\end{figure}

A {\em generalized patch} (of fullerene) is  a connected plane graph where all faces are
hexagons and pentagons except one outer face, with vertices not on the outer face having
degree 3 and vertices on the outer face having degree 1, 2 or 3.

\begin{Lemma}\label{induction}Every  $F_i''$ is a generalized patch of fullerene. Further, if $F_i'$ has a pendent vertex, then at least one pendent vertex is incident with two edges in $X_i$ that are consecutive along the boundary of $F_i''$.
\end{Lemma}
\begin{proof}From the above discussions we have that all  $F_i''$ are generalized patches from $i=1$ to 3 and 4  respectively in above two cases.  Suppose that $F_i''$ is a generalized patch and  $F_i'$ has a pendent vertex. Then $F_i'$ lies in the outer face of $F_i''$.  Let $x$ be any pendent vertex of  $F_i'$ and let $y_1$, $y_2$ and $y_3$ be three neighbors of $x$ such that $y_1,y_2\in F_i''$ and $y_3\in F_i'$. If $F_i'$ is connected, then edges $xy_1$ and $xy_2$ and a part of the boundary of $F_i''$ bound a region that is a face of $F$. This implies that $xy_1$ and $xy_2$ are consecutive along the boundary of $F_i''$. Further if $y_3$ has one neighbor $y_4$ in $F_i''$, then similarly we have that $y_3y_4$, $xy_1$ and $xy_2$ are consecutive in $X_i$ along the boundary of $F_i''$. If $y_3$ has two  neighbors  in $F_i''$, then $F_{i+1}''=F$.  Therefore $F_{i+1}''$ is a generalized patch. If $F_i'$ is disconnected, then it has exactly two components that each contains a unique perfect matching by $|X_i|\leq 8$. Since $c\lambda(F)=5$,  by Lemmas \ref{3-edge-cut} and \ref{4-edge-cut} both components of  $F'_i$ are complete graphs with two vertices. So we can see that at least one pendent vertex $F'_i$ is incident with two edges in $X_i$ that are consecutive along the boundary of $F_i''$ and $F_{i+1}''$ is a generalize patch. This induction establishes the lemma. 
\end{proof}

Since configurations $F_{hj}$, $1\leq j\leq 5$, in Figs. \ref{fdaofk} and \ref{fdaofk2} can be generated from some of configurations $F_{si}$, $1\leq i\leq 15$, by merging two half edges into one edge. So let us call the configurations $F_{s1} F_{s2},\ldots,F_{s15}$  the \emph{initial seed graphs}. Our construction idea is that starting from such initial seed graphs to expand incrementally a seed graph $S$ to a large {\em  seed graph} $S'$  by using several simple operations on the boundary of the seed graphs to realize the growing of generalized patches $F_i''$ (Lemma \ref{induction}).

Similarly to the case of using the boundary code to describe a fullerene patch, we use a sequence to describe the boundary of a seed graph $S$. We label clockwise (counterclockwise) the half-edges of $S$ by $t_1,t_2,\ldots,t_k$, and set $a_i$ as the number of vertices from $t_i$ to $t_{i+1}$ in a clockwise (counterclockwise) scan of the boundary of $S$. We call then the cyclic sequence $[a_1,a_2,\ldots,a_k]$ a \emph{distance-array} of $S$. For instance, a distance-array to describe the boundary of $F_{S1}$ shown in Figure $8$ could be $[12222216]$. Since a fullerene graph has only pentagonal and hexagonal faces, we have that $1\le a_i\le 6$.  Note that a boundary may have more than one distance-arrays describing it since we might start reading the boundary from different position and we could read the boundary in clock or counter clock direction. However, it is easy to see that for the same boundary, the boundary arrays are rotations and reversions of each other and we consider them equivalent in this paper. Thus we shall make no distinction between the boundary of a seed graph and a distance-array describing it.

We define the following operations on a seed graph $S$:

\begin{itemize}
\item[$(O_1)$]  If the length of a distance-array of $S$ is at least $4$, $a_i$ is $4$ or $5$ and $a_{i-1}$ and $a_{i+1}$ are both at most $4$,  let $t_i$ and $t_{i+1}$ be incident to a new vertex $u$. Add another new vertex $v$, an edge $uv$, and attach two half edges to $v$. A distance array for the resulting seed graph $S'$ would be $[a_1,\ldots,a_{i-2},a_{i-1}+2,1,a_{i+1}+2,a_{i+2},\ldots a_k]$.

\item[$(O_2)$] If the length of a distance-array of $S$ is at least $4$, $a_i$ is $5$ or $6$ and $a_{i-1}+a_{i+1}$ is at most $6$, let the half edges $t_i$ and $t_{i+1}$ merge into one edge. A distance array for the resulting seed graph $S'$ would be $[a_1,\ldots,a_{i-2},a_{i-1}+a_{i+1}, a_{i+2},\ldots, a_k]$.

%

\item[$(O_3)$] If  the length of a distance-array of $S$ is $2$ and $a_1$ and $a_2$ are both $5$ or $6$, we merge the half edges $t_1$ and $t_2$ into one edge. The distance-array corresponding to the resulting graph $S'$ is the empty distance-array $[]$.

\item[$(O_4)$] If the length of a distance-array of $S$ is $8$, $a_i$ and $a_{i+4}$ are both $1$ or $2$ for some $i$, and all other $a_j$ are $3$ or $4$, we connect each half-edge of $S$ to a vertex of degree $2$ in the graph shown in Figure $4$, in the only admissible way. The distance-array corresponding to the resulting graph $S'$ is the empty distance-array $[]$.
\end{itemize}

For  example, see Fig.  \ref{12}. It is not difficult to see that the operations above maintain the desired property. With the exception of $O_4$, they are primitive operations in the sense that they add the minimum number of vertices, edges or half-edges necessary to generate larger seed graphs.

\begin{figure}[htbp]
\begin{center}
\includegraphics[totalheight=5.5 cm]{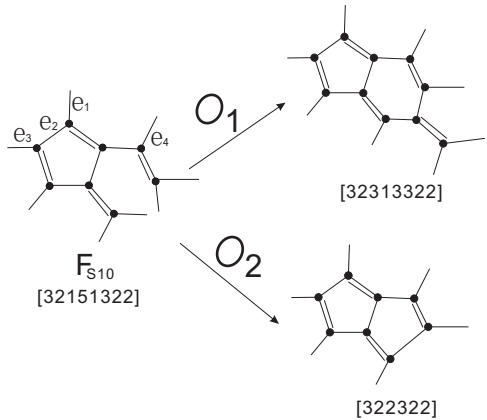}
\caption{\label{12}\small{An example for  two operations $O_1$ and $O_2$ on
 $F_{s10}$  and on the
 distance-array.}}
\end{center}
\end{figure}

Operation $O_1$ adds two vertices to a seed graph and preserves the length of the distance array. Whereas, Operations $O_2$ and $O_3$ preserve the number of vertices in the seed graph but reduce by $2$ the length of the distance-array. Operation $O_4$  adds  $10$  vertices to the seed graph. Clearly,  the application of Operations $O_3$ and $O_4$ produces a fullerene graph, and no further operation can be performed so they are \emph{terminating} operations.

The next stage is to compute a directed graph $D$ representing all possible ways to obtain fullerene graphs  from the initial seed graphs by the successive implementations (in any order and a finite number of times) of the four operations $O_1$ to $O_4$. We call this directed graph $D$ the \emph{distance-array digraph}. The vertices in $D$ will be the distance-arrays of all conceivable seed graphs, and
there will be an arc $(s, s')$ in $D$ if and only if there is an operation on a seed graph $S$ with boundary $s$ yielding a seed graph $S'$ with boundary $s'$.

Let $L_{i}$ denote the distance-array determined by  $F_{Si}$, $1\leq
i\leq15$. Then we have

 $L_{1}=[12222216]$, $L_{2}=[13222152]$, $L_{3}=[12322143]$,
$L_{4}=[14221422]$,

$L_{5}=[12421323]$, $L_{6}=[13321332]$, $L_{7}=[13215123]$,
$L_{8}=[14123214]$,

 $L_{9}=[14124123]$, $L_{10}=[32151322]$, $L_{11}=[41241322]$, $L_{12}=[41232313]$,

$L_{13}=[32142313]$, $L_{14}=[25122322]$, $L_{15}=[24213322]$.

From such initial distance arrays we describe the following procedure to generate $D$.\\

\begin{Alg}[Generation of Distance-Array Digraph $D$]\hfill

\begin{itemize}
\item[$(S1)$] Set $V=\{L_1,L_2,\ldots,L_{15}\}$ and $A=\emptyset$.
\item[$(S2)$] Select a distance array $s\in V$ on which no operations have been made.

\item[$(S3)$]For a suitable operation ($O_1$ to $O_4$) on $s$, make it  to get a distance array $s'$, then  set $V:=V\cup \{s'\}$ whenever  $s'\not\in V$, and set $A:=A\cup \{(s,s')\}$. Repeat this procedure for each of such operations. If no suitable operations on $s$, we say we have made operations on $s$.
\item[$(S4)$] If all distance arrays in $V$ have been selected to make operations, then go to $(S5)$. Otherwise, go to (S2).
\item[$(S5)$] For every vertex $s\in V$, delete $s$ if there is no directed  path from $s$ to the empty distance-array $[]$ (at this point we know $[]\in V$). Then obtain the final directed graph $D$ with vertex-set $V$ and arc-set $A$.
\end{itemize}
\end{Alg}
\begin{figure}[htbp]
\begin{center}
\includegraphics[totalheight=6.2 cm]{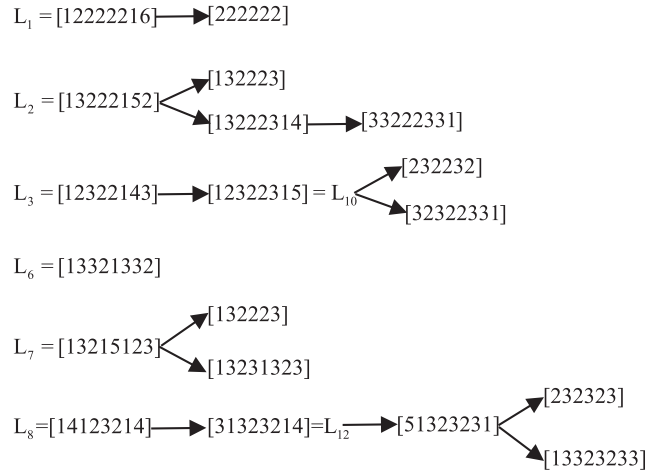}
\caption{\label{fig10}\small{Generation of all possible distance arrays from $L_1$, $L_2$, $L_3$, $L_6$, $L_7$, $L_8$, $L_{10}$ and $L_{12}$. }}
\end{center}
\end{figure}

We now use the above procedure to generate $D$. From the initial distance arrays $L_1$, $L_2$, $L_3$, $L_6$,$L_7$, $L_8$, $L_{10}$ and $L_{12}$, we implement step (S3) repeatedly for all possible operations (in fact $O_1$ and $O_2$) to produce a digraph on a series of distance arrays, see Fig. \ref{fig10}. We can see that they only reach ``dead" distance arrays (i.e. non-empty distance arrays for which no further operations can be made). So all such distance arrays are discarded and does not appear in the final $D$.

For the remaining initial distance arrays,  from them we also implement (S3) repeatedly by  all possible operations  to produce a digraph on a series of distance arrays, and  discard all distance arrays that cannot reach the empty distance array. We can check that the resulting directed graph  $D$ is as shown in Fig. \ref{final-D} . It can be also verified by a program.

\begin{figure}[htbp]
\begin{center}
\includegraphics[totalheight=13.5 cm]{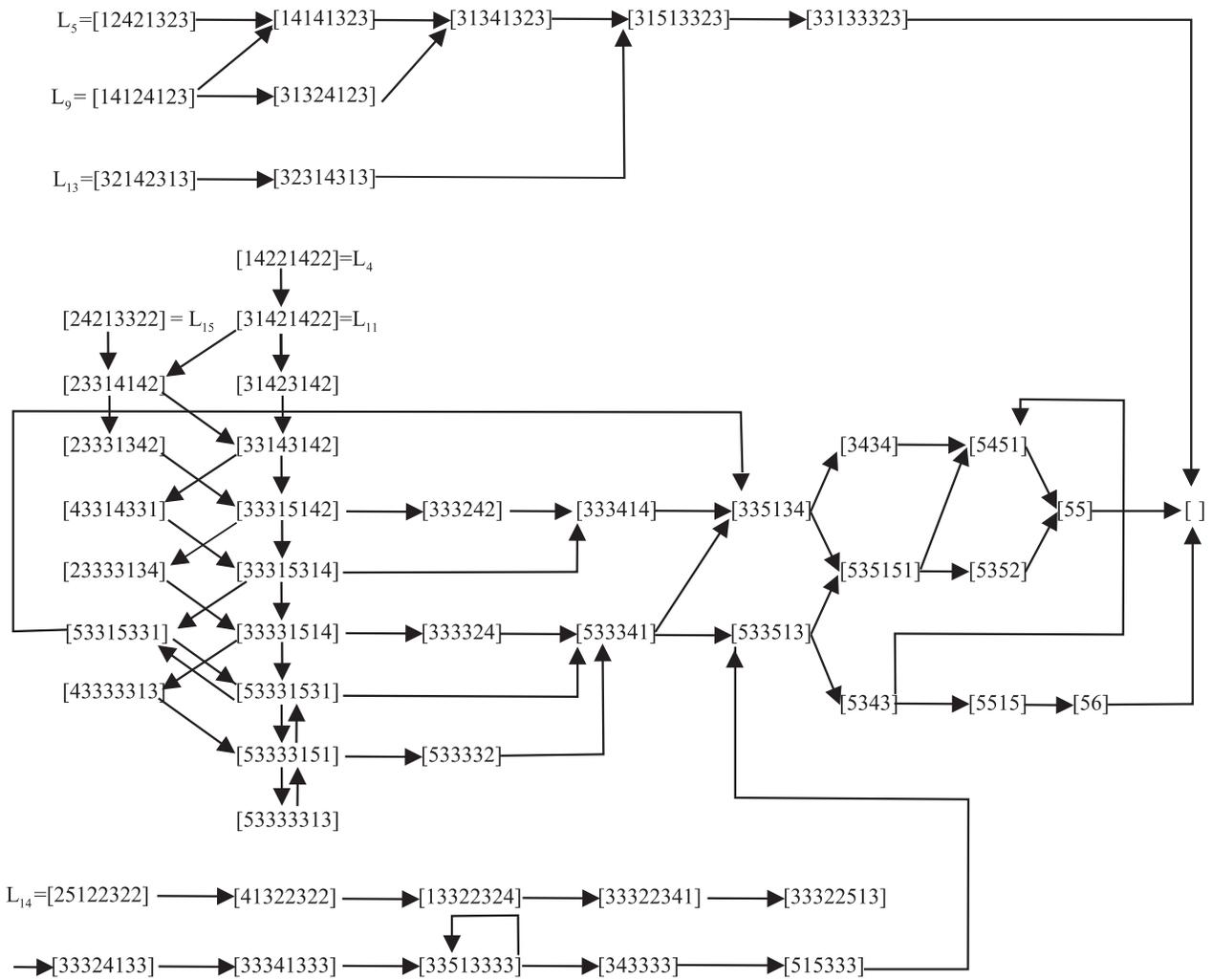}
\caption{\label{final-D}\small{Distance-Array Digraph $D$. }}
\end{center}
\end{figure}

After the above computations, the distance-arrays digraph $D$ turned out to have $52$ vertices and  $72$ arcs. We found that $D$ contains exactly one loop, at vertex [3351333], and three pair of symmetric arcs, which form some possible closed directed walks. 

Given a sequence of distance arrays $A_1,A_2,\ldots,A_{t+1}$ such that an operation $O^i$ on $A_i$ is made to get $A_{i+1}$ for each $1\leq i\leq t$, where $O^i\in\{O_1,O_2,O_3,O_4\}$. If $A_1$ is the distance array of a seed graph $S_1$, then we accordingly obtain the sequence of seed graphs $S_1,S_2,\ldots,S_{t+1}$ by implementing the same series of operations from $S_1$. That is, for each $1\leq i\leq t$ $S_{i+1}$ has the distance array $A_{i+1}$ and $S_{i+1}$ is obtained from $S_i$ via operation $O^i$. We may say {\em $S_{t+1}$ is generated from $S_1$ by making a series of operations along  sequence  $A_1,A_2,\ldots,A_{t+1}$.}

For example, take any directed path in $D$ from $L_5$, $L_9$ or $L_{13}$ to [33133323]. From the corresponding initial seed graph we generate a seed graph  by making a series of operations along  this directed path, then make the terminating operation $O_4$ on it to always get fullerene graph $F_{24}$.

Take another directed path in $D$: $L_{15}=[24213322]\rightarrow[23314142]\rightarrow[33143142]\rightarrow[33315142]\rightarrow[333242]\rightarrow[333414]\rightarrow[335134]\rightarrow[3434]\rightarrow[5451]\rightarrow[55]\rightarrow[]$.
Since the initial seed graph $F_{s15}$ has eight vertices,  from $F_{s15}$ we generate fullerene graph $F_{20}$ by making a series of operations along this directed path.

We now describe  our method generating all fullerene graphs with anti-forcing number 4 as follows.
\begin{thm}\label{generate}
A fullerene graph has the anti-forcing number 4 if and only if it can be generated from one of six initial seed graphs $F_{s4}, F_{s5}, F_{s9}, F_{s13}, F_{s14}$, $F_{s15}$ by making a series of operations along a directed walk in $D$ from its distance array to the empty distance array.
\end{thm}

\begin{proof}If a fullerene graph $F$ has the anti-forcing number 4, then by the above discussions  $F$ can reconstructed from some initial seed graph $F_{si}$ by making a series of operations in $O_1$ to $O_4$, which also correspond to a directed walk in $D$ from  the distance array of $F_{si}$ to the empty distance array. 

Conversely, let $F$ be a fullerene graph generated from some initial seed graphs $F_{si}$ by making a series of operations along a directed walk from its distance array to the empty distance array. Let $E_0:=\{e_1,e_2,e_3,e_4\}$ relating with $F_{si}$. It suffices to show that $F-E_0$ has a unique perfect matching. We can check that the vertices of $F_{si}$ have a unique pairing in $F-E_0$ (see double edges in Figs. \ref{fdaofk} and \ref{fdaofk2}). That is, any perfect matching of $F-E_0$ (in fact it  exists) has the restriction on $F_{si}$, which is its unique perfect matching. For any middle seed graph $S$ generated from $F_{si}$ along a directed walk, suppose that any perfect matching of $F-E_0$ has the restriction on $S-E_0$ that is its unique perfect matching. If operation $O_2$ or $O_3$ is made to get a seed graph $S'$, then since $V(S)=V(S')$ the same statement holds also for $S'-E_0$. If operation $O_1$  is made to get a seed graph $S'$, then one added vertex has two neighbors in $S$. Hence it is incident with a pendant edge in $F-S$, which must belong to a perfect matching of $F-E_0$, and the same statement still holds for $S'-E_0$. If operation $O_4$  is made to get a seed graph $S'$, then by Theorem \ref{F' structure} $F-S$ is a  subgraph formed by two disjoint pentagons connected by an edge, and $F-S$ has a unique perfect matching.  In this situation $F=S'$ and  $F-E_0$ has a unique perfect matching. This induction establishes the required fact.
\end{proof}


\begin{thm}
For any even $n\ge20$ ($n\neq 22,26$) there is a fullerene $F_n$ such that $af(F_n)=4$.
\end{thm}

\begin{proof}

Fig. \ref{Example} already covers the cases $n=20$ and $n=24$, so we assume $n=2k$ with $k\ge 14$.

Consider a directed walk in $D$:  $L_{14}=[25122322]\rightarrow[41322322]\rightarrow[13322324]\rightarrow[33322341]\rightarrow[33322513]\rightarrow[33324133]
\rightarrow[33341333]\rightarrow[33513333]^{k-13}\rightarrow[343333]\rightarrow[515333]\rightarrow[533513]\rightarrow[5343]
\rightarrow[5451]\rightarrow[55]\rightarrow[]$, where $[33513333]^{k-13}$ means $k-13$ repetitions of  [33513333] and $k-14$ repetitions of a loop at [33513333].   Recall that an operation preserving the length of the distance-array increases by $2$ the number of vertices in the seed graph, and an operation decreasing by $2$ the length of the distance array preserves the number of vertices. Along the directed walk from $L_{14}$ to [] we can see that operation $O_1$ is made $k-4$ times.  Since the initial graph has eight vertices, from it we generate a fullerene graph $F_{2k}$ with $2k (=8+2(k-4))$ vertices along this directed path. By Theorem \ref{generate} we have that this fullerene graph $F_{2k}$ has the anti-forcing number 4.
\end{proof}

%


\begin{cor}
$af(F_{26})=5$.
\end{cor}

\begin{proof}
From the above  constructions for $F_{24}$ and $F_{2k}$ $(k\geq 14)$, we know that any directed path from $L_5,L_9$,or $L_{13}$ to [] in $D$ only produce fullerene $F_{24}$, and any directed walk from $L_{14}$ to [] in $D$ only produce fullerenes with at least 28 vertices. So we only consider all directed walks from $L_4,L_{11}$, and  $L_{15}$ to  [] in $D$, which form a directed subgraph $D'$ of $D$. We can check that $D'$ is a {\em bipartite} directed graph. Let $P$ denote any directed walk of length $l$ from $L_{11}$ or  $L_{15}$ to []. Then from the initial seed graphs $F_{s11}$ or $F_{s15}$ we obtain a fullerene $F$ by making a series of operations along $P$, and $F$ has  $8+2(l-4)=2l$ vertices. Further, there are directed paths of length 10 from $L_{11}$ and   $L_{15}$ to [] in  $D'$ respectively. Since $D'$ is bipartite, $P$ has the length with the same parity with $10$. So $l$ is even and $F$ must have $4k$ vertices. From such discussions, by Theorem \ref{generate} we have $af(F_{26})\geq 5.$ On the other hand, there is a unique fullerene with 26 vertices and  we find an
 anti-forcing set with the size of 5 (see Fig. \ref{f26}). So $af(F_{26})=5$.
\end{proof}

\begin{figure}[htbp]
\begin{center}
\includegraphics[totalheight=5 cm]{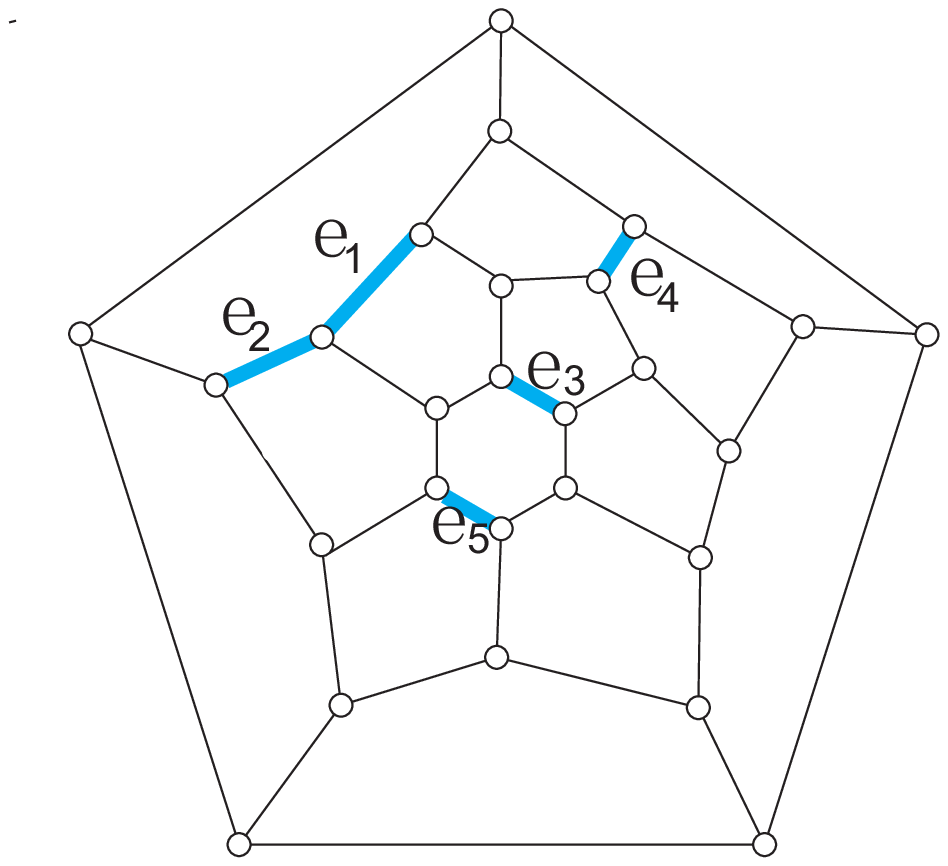}
\caption{\label{f26}\small{An anti-forcing set with the size of 5 on
$F_{26}$. }}
\end{center}
\end{figure}


\end{document}